\documentclass[12pt]{amsart}

\usepackage{amsmath}
\usepackage{amssymb}
\usepackage{graphicx}
\usepackage{amscd}
\usepackage[all]{xy}

\newtheorem{theorem}{Theorem}[section]
\newtheorem{proposition}[theorem]{Proposition}
\newtheorem{lemma}[theorem]{Lemma}

\theoremstyle{definition}
\newtheorem{definition}[theorem]{Definition}
\newtheorem{example}[theorem]{Example}
\newtheorem{remark}[theorem]{Remark}

\author{Takuya Sakasai}
\address{Graduate School of Mathematical Sciences, 
The University of Tokyo, 
3-8-1 Komaba Meguro-ku Tokyo 153-8914, Japan}
\email{sakasai@ms.u-tokyo.ac.jp}


\newcommand{\Ker}{\mathop{\mathrm{Ker}}\nolimits}
\newcommand{\Hom}{\mathop{\mathrm{Hom}}\nolimits}
\newcommand{\Id}{\mathop{\mathrm{id}}\nolimits}
\newcommand{\Z}{\ensuremath{\mathbb{Z}}}
\newcommand{\Tmatrix}[1]{\mathop{\left( {#1} \right)}\nolimits}
\newcommand{\Aut}{\mathop{\mathrm{Aut}}\nolimits}

\newcommand{\Cg}{\ensuremath{\mathcal{C}_{g,1}}}

\newcommand{\Acy}{F^{\mathrm{acy}}}
\newcommand{\Alg}{F^{\mathrm{alg}}}
\newcommand{\AC}[1]{#1^{\mathrm{acy}}}

\title[The Magnus representation and homology cylinders]
{The Magnus representation and homology cobordism 
groups of homology cylinders}

\begin{document}

\begin{abstract}
A {\it homology cylinder\/} over a compact manifold is a homology cobordism between 
two copies of the manifold together with a boundary parametrization. 
We study abelian quotients of the homology cobordism group of homology cylinders. 
For homology cylinders over general surfaces, it was shown by 
Cha, Friedl and Kim that their homology cobordism groups 
have infinitely generated abelian quotient groups by using Reidemeister torsion invariants. 
In this paper, we first investigate their abelian quotients again by using 
another invariant called the Magnus representation. After that, we apply 
the machinery obtained from the Magnus representation to higher dimensional cases and show that 
the homology cobordism groups of homology cylinders 
over a certain series of manifolds regarded as a generalization of surfaces have big abelian quotients. 
In the proof, a homological localization, called 
the {\it acyclic closure}, of a free group and its 
automorphism group play important roles and 
our result also provides some information on these groups 
from a group-theoretical point of view. 
\end{abstract}
% About 165 words

\maketitle

%----------Introduction--------------

\section{Introduction}\label{sec:intro}
Let $\Sigma_{g,1}$ be a compact connected oriented surface of genus $g$ 
with one boundary component. 
A {\it homology cylinder} over $\Sigma_{g,1}$ 
is a homology cobordism between two copies of $\Sigma_{g,1}$ 
together with a boundary parametrization. 
The set of all isomorphism classes of 
homology cylinders has a natural product operation given by 
stacking, so that it forms a monoid denoted by $\mathcal{C}_{g,1}$. 
The study of the monoid $\mathcal{C}_{g,1}$ was initiated by 
Goussarov \cite{gou} and Habiro \cite{habiro} in their theory 
of clasper surgery and finite type invariants of 3-dimensional manifolds. 
In their study, a quotient group of the monoid by clasper surgery equivalence 
was introduced and 
its structure was intensively clarified in Massuyeau-Meilhan \cite{mm1,mm2} and 
Habiro-Massuyeau \cite{hm}. 

On the other hand,  Garoufalidis and Levine \cite{gl} introduced another 
equivalence relation by considering 
homology cobordisms of homology cylinders and they defined the 
quotient group $\mathcal{H}_{g,1}$ called 
the {\it homology cobordism group} of homology cylinders. 
Investigating the structure of the group $\mathcal{H}_{g,1}$ is 
considered to be important because 
it provides an efficient way of understanding the set of homology cobordism classes of 
3-dimensional manifolds as the braid group contributes to 
knot theory. In this paper, we focus on this group $\mathcal{H}_{g,1}$ 
as well as its generalization to higher dimensional cases. 

Another motivation for studying the monoid and 
groups of homology cylinders comes from the 
fact that they include the mapping class group $\mathcal{M}_{g,1}$ of $\Sigma_{g,1}$. 
They share many properties. 
For example, Garoufalidis-Levine \cite{gl}, Levine \cite{levine} 
and Habegger \cite{habe} gave a deep relationship to the theory of 
Johnson homomorphisms used originally 
in the study of $\mathcal{M}_{g,1}$ and its subgroups.  
In pursuing more relationships between 
$\mathcal{M}_{g,1}$ and $\mathcal{H}_{g,1}$, 
it should be an important step to determine and 
compare their abelianizations. As for $\mathcal{M}_{g,1}$, 
it was first shown by Harer \cite{harer} that 
the abelianization is trivial (namely, the group $\mathcal{M}_{g,1}$ is perfect) 
except a few low genus cases. On the other hand, the abelianization of $\mathcal{H}_{g,1}$ 
has not yet been determined (see Section \ref{sec:abelian} for details). 

In our previous paper \cite{sakasai08}, we introduced two kinds 
of invariants of homology cylinders, the Magnus representation and 
the Reidemeister torsion. 
Both invariants are crossed homomorphisms from the monoid $\mathcal{C}_{g,1}$ 
to some groups of matrices. It was observed that 
the Magnus representation factors through $\mathcal{H}_{g,1}$, 
while the Reidemeister torsion does not so. 
In the same paper, we found many abelian quotients of {\it sub}monoids of $\mathcal{C}_{g,1}$ and 
{\it sub}groups of $\mathcal{H}_{g,1}$. However 
no information was extracted on abelian quotients of the whole monoid and group.  
It had been conjectured that $\mathcal{H}_{g,1}$ was perfect for general $g$ 
as in the case of $\mathcal{M}_{g,1}$. 

After that, however, Cha, Friedl and Kim \cite{cfk} 
succeeded in showing that the abelianization of $\mathcal{H}_{g,1}$ is infinitely generated 
by using a version of the Reidemeister torsion which we will call the {\it $H$-torsion\/} in this paper. 
In fact, they took an appropriate reduction of the torsion invariant so that the 
resulting map becomes a homology cobordism invariant homomorphism. 

The purpose of the first half of this paper is to use (the determinant of) the Magnus 
representation together with 
Cha-Friedl-Kim's reduction technique to investigate again abelian quotients of $\mathcal{H}_{g,1}$. 
In Section \ref{sec:abelian}, we will show that there exists a relationship between 
the invariant obtained from the Magnus representation and Cha-Friedl-Kim's torsion invariant.  

Our invariant using the Magnus representation can be easily applied to 
homology cylinders over higher dimensional manifolds, 
which seem to have their own interest. 
Investigating this invariant is the purpose of the second half. 
The main theorem (Theorem \ref{thm:acyMag}) is that 
the homology cobordism groups of 
homology cylinders over a certain series of manifolds regarded as a higher dimensional 
generalization of surfaces have 
abelian quotients isomorphic to the free abelian group of infinite rank. 
For the proof, we use a purely group-theoretical description of 
the Magnus representation 
as the representation of the automorphism group of 
the {\it acyclic closure} of a free group. 
This description was first given by Le Dimet \cite{ld} 
in the context of string links, and then 
we clarified its relationship to homology cylinders in our previous paper \cite{sa2}. 
The definition of the acyclic closure, which is originally due to 
Levine \cite{le1, le2}, and its 
fundamental properties are reviewed
in Section \ref{sec:acyclic}. We prove the theorem by constructing first 
an epimorphism from the homology cobordism group onto 
the automorphism group of the acyclic closure of a free group  (Theorem \ref{thm:acysurj})
and then showing that this automorphism group has an abelian quotient isomorphic to 
the free abelian group of infinite rank (Theorem \ref{thm:H1Acy}). 

All manifolds are assumed to be smooth 
throughout this paper, while 
similar statements hold for other categories. 
We use the same notation to write a continuous map and the 
induced homomorphisms on fundamental groups and homology groups. 
Also, all homology groups are with $\mathbb{Z}$-coefficients.

The author would like to thank Yasushi Kasahara and 
Gw\'ena\"el Massuyeau for helpful comments and discussions. 
This research was partially supported by JSPS KAKENHI 
(No.~21740044 and No.~24740040), 
Japan Society for the Promotion of Science, Japan. 
%Ministry of Education, Science, Sports and Technology, Japan.

%----------Homology cylinders over a manifold--------------

\section{Homology cylinders over a manifold}\label{sec:HCgeneral}

We begin by giving the definition of homology cylinders over a manifold. 
Originally, homology cylinders are introduced and studied for surfaces by 
Goussarov \cite{gou}, Habiro \cite{habiro}, 
Garoufalidis and Levine \cite{gl, levine}. 
The definition below is a natural generalization to it. 

Let $X$ be a compact oriented connected $k$-dimensional manifold. 
We assume for simplicity 
that the boundary $\partial X$ of $X$ is connected or empty. 

\begin{definition}\label{def:HCgeneral}
A {\it homology cylinder\/} {\it over} $X$   
consists of a compact oriented $(k+1)$-dimensional manifold $M$ 
with two embeddings $i_{+}, i_{-}: X \hookrightarrow \partial M$ 
such that:
\begin{enumerate}
\renewcommand{\labelenumi}{(\roman{enumi})}
\item
$i_{+}$ is orientation-preserving and $i_{-}$ is orientation-reversing, 
\item 
$\partial M=i_{+}(X)\cup i_{-}(X)$ and 
$i_{+}(X)\cap i_{-}(X)
=i_{+}(\partial X)=i_{-}(\partial X)$,
\item
$i_{+}|_{\partial X}=i_{-}|_{\partial X}$, 
\item
$i_{+},i_{-} : H_{*}(X)\to H_{*}(M)$ are isomorphisms. 
\end{enumerate}
We denote a homology cylinder by $(M,i_{+},i_{-})$ or simply $M$. 
The boundary $\partial M$ of $M$ is the double of $X$ if $\partial X \neq \emptyset$. Otherwise it is 
the disjoint union of two copies of $X$.  
\end{definition}

Two homology cylinders $(M,i_+,i_-)$ and $(N,j_+,j_-)$ over $X$ 
are said to be {\it isomorphic} if there exists 
an orientation-preserving diffeomorphism $f:M \xrightarrow{\cong} N$ 
satisfying $j_+ = f \circ i_+$ and $j_- = f \circ i_-$. 
We denote by $\mathcal{C} (X)$ the set of all isomorphism classes 
of homology cylinders over $X$. 
We define a product operation on $X$ by 
\[(M,i_+,i_-) \cdot (N,j_+,j_-)
:=(M \cup_{i_- \circ (j_+)^{-1}} N, i_+,j_-)\]
for $(M,i_+,i_-)$, $(N,j_+,j_-) \in \mathcal{C} (X)$, 
which endows $\mathcal{C} (X)$ with a monoid 
\index{monoid of homology cylinders} structure. 
The unit is the trivial homology cylinder 
$(X \times [0,1], \mathrm{id} \times 1, \mathrm{id} \times 0)$, 
where collars of $i_+ (X)=(\mathrm{id} \times 1) (X)$ 
and $i_- (X)=(\mathrm{id} \times 0)(X)$ 
are stretched half-way along $(\partial X) \times [0,1]$ 
so that $i_+ (\partial X)=i_- (\partial X)$. 

\begin{example}\label{ex:diffeotopy}
For a self-diffeomorphism $\varphi$ of 
$X$ which restricts to the identity map on a neighborhood of 
$\partial X$, we can construct a homology cylinder by setting 
\[(X \times [0,1], \mathrm{id} \times 1, \varphi \times 0)\]
with the same treatment of the boundary as above. 
It is easily checked that the isomorphism class of 
$(X \times [0,1], \mathrm{id} \times 1, \varphi \times 0)$ 
depends only on the isotopy (fixing a neighborhood of $\partial X$ pointwise) 
class of $\varphi$ and that 
this construction gives a monoid homomorphism 
from the diffeotopy group $\mathcal{M} (X)$ 
to $\mathcal{C} (X)$. 
\end{example}
\begin{remark}
The homomorphism $\mathcal{M} (X) \to \mathcal{C} (X)$ 
is {\it not} necessarily injective. 
In fact, if $[\varphi] \in \mathrm{Ker}\,(\mathcal{M} (X) \to \mathcal{C} (X))$, 
the definition of the homomorphism only says that 
$\varphi$ is a {\it pseudo isotopy} over $X$. 
\end{remark}
We also introduce {\it homology cobordisms} 
of homology cylinders, which define an equivalence relation 
among homology cylinders. 
\begin{definition}
Two homology cylinders $(M,i_+,i_-)$ and $(N,i_+,i_-)$ over 
$X$ are said to be {\it homology cobordant} 
if there exists a compact oriented $(k+2)$-dimensional manifold $W$ such that: 
\begin{enumerate}
\item $\partial W = M \cup (-N) /(i_+ (x)= j_+(x) , \,
i_- (x)=j_-(x)) \quad x \in X$, 
\item the inclusions $M \hookrightarrow W$, $N \hookrightarrow W$ 
induce isomorphisms on the homology group,  
\end{enumerate}
where $-N$ denotes the manifold $N$ with the opposite orientation.  
\end{definition}
\noindent
We denote by $\mathcal{H} (X)$ 
the quotient set of $\mathcal{C} (X)$ with respect 
to the equivalence relation of homology cobordism. 
The monoid structure of $\mathcal{C} (X)$ induces 
a group structure of $\mathcal{H} (X)$. 
We call $\mathcal{H} (X)$ the {\it homology cobordism group} of 
homology cylinders over $X$.

%----------Homology cylinders over a surface--------------

\section{Homology cylinders over a surface}\label{sec:HC}
Let $\Sigma_{g,1}$ be a compact oriented surface of genus $g$ 
with one boundary component. 
We take a base point $p$ of $\Sigma_{g,1}$ on the boundary 
$\partial \Sigma_{g,1}$ and $2g$ oriented loops 
$\gamma_1, \gamma_2, \ldots, \gamma_{2g}$ 
as in Figure \ref{fig:generator}. These loops form 
a spine $R_{2g}$ of 
$\Sigma_{g,1}$ and they give 
a basis of $\pi_1 (\Sigma_{g,1})$, a free group of rank $2g$. 
The boundary loop $\zeta$ is given by 
$\zeta=[\gamma_1, \gamma_{g+1}][\gamma_2, \gamma_{g+2}] \cdots 
[\gamma_g, \gamma_{2g}]$. 
We denote the first homology group $H_1 (\Sigma_{g,1})$ by $H$ for simplicity. 
The group $H$ can be identified with $\mathbb{Z}^{2g}$ by choosing 
$\{\gamma_1, \gamma_2, \ldots, \gamma_{2g}\}$ as a basis of $H$, where 
we write $\gamma_j$ again for $\gamma_j$ as an element of $H$. 
This basis is a symplectic basis with respect to the intersection pairing 
on $H$. 

\begin{figure}[htbp]
\begin{center}
\includegraphics{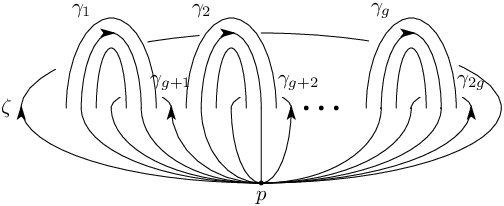}
\end{center}
\caption{Our basis of $\pi_1 (\Sigma_{g,1})$}
\label{fig:generator}
\end{figure}

We use the notation $\mathcal{M}_{g,1}:=\mathcal{M} (\Sigma_{g,1})$, 
$\mathcal{C}_{g,1}:=\mathcal{C} (\Sigma_{g,1})$ and $\mathcal{H}_{g,1}:=\mathcal{H} (\Sigma_{g,1})$ 
following our previous papers. The diffeotopy group $\mathcal{M}_{g,1}$ is also called the {\it mapping class group\/} 
of $\Sigma_{g,1}$. It was shown by Garoufalidis-Levine \cite[Section 2.4]{gl} that 
the homomorphism $\mathcal{M}_{g,1} \to \mathcal{C}_{g,1}$ and the composition 
$\mathcal{M}_{g,1} \to \mathcal{C}_{g,1} \to \mathcal{H}_{g,1}$ are injective. 

\begin{example}[Levine \cite{levine}]\label{ex:stringlink}
Let $L$ be a string link of $g$ strings, 
which is a generalization of a pure braid. 
We embed a $g$-holed disk $D_g^2$ into $\Sigma_{g,1}$ as a closed 
regular neighborhood of the union of the 
loops $\gamma_{g+1}, \gamma_{g+2}, \ldots, \gamma_{2g}$ in Figure \ref{fig:generator}. 
Let $C$ be the complement of an open tubular neighborhood of $L$ in 
$D^2 \times [0,1]$. By choosing a framing of $L$, 
we can fix a diffeomorphism $h:\partial C \xrightarrow{\cong} 
\partial (D_g^2 \times [0,1])$. 
Then the manifold $M_L$ obtained from $\Sigma_{g,1} \times [0,1]$ 
by removing $D_g^2 \times [0,1]$ and regluing 
$C$ by $h$ becomes a homology cylinder with the same boundary parametrizations 
$i_+$, $i_-$ as the trivial homology cylinder. 
\end{example}

The monoid $\mathcal{C}_{g,1}$ and the group $\mathcal{H}_{g,1}$ 
share many properties with the group $\mathcal{M}_{g,1}$. 
The most fundamental one is given by their action on $H$. 
Define a map 
\[\sigma: \mathcal{C}_{g,1} \longrightarrow \mathrm{Aut}\,(H)\]
by assigning to $(M,i_+,i_-) \in \mathcal{C}_{g,1}$ the 
automorphism $i_+^{-1} \circ i_-$ of $H$. 
This map extends the natural action of $\mathcal{M}_{g,1}$ 
on $H$ and it is a monoid homomorphism. 
The image of $\varphi$ consists of 
the automorphisms of $H$ preserving 
the intersection pairing. 
Therefore, under the identification 
$H \cong \mathbb{Z}^{2g}$ mentioned above, 
we have an epimorphism 
\[\sigma: \mathcal{C}_{g,1} \longrightarrow 
\mathrm{Sp}(2g, \mathbb{Z}).\]
We put $\mathcal{IC}_{g,1}:=\mathrm{Ker}\,\sigma$, which is 
an analogue of the Torelli group 
$\mathcal{I}_{g,1}=\mathrm{Ker} (\sigma: \mathcal{M}_{g,1} \to 
\mathrm{Sp}(2g, \mathbb{Z}))$. We can see that $\sigma$ induces a 
group homomorphism $\sigma:\mathcal{H}_{g,1} \to \mathrm{Sp}(2g, \mathbb{Z})$ 
and we denote its kernel by $\mathcal{IH}_{g,1}$.

%----------Magnus representation and $H$-torsion for homology cylinders--------------

\section{Magnus representation and $H$-torsion for homology cylinders}
Here, we recall two kinds of invariants for homology cylinders 
from \cite{gs08, sakasai08}, which are analogous to invariants 
for string links defined by Le Dimet \cite{ld} and 
Kirk-Livingston-Wang \cite{klw}. 

Since the group $H=H_1 (\Sigma_{g,1})$ is free abelian, 
the group ring $\mathbb{Z} [H]$ is isomorphic to the Laurent polynomial ring of 
variables $\gamma_1, \gamma_2, \ldots, \gamma_{2g}$. 
We can embed $\mathbb{Z} [H]$ into  
the fractional field $\mathcal{K}_H :=
\mathbb{Z} [H] (\mathbb{Z} [H] - \{0\})^{-1}$. 

Let $(M,i_+,i_-) \in \mathcal{C}_{g,1}$ be a homology cylinder. 
Since $H_1 (M) \cong H_1 (\Sigma_{g,1})$, 
the field $\mathcal{K}_{H_1 (M)}:=\mathbb{Z} [H_1 (M)] 
(\mathbb{Z} [H_1 (M)]- \{0\})^{-1}$ 
is defined. 
We regard $\mathcal{K}_H$ and $\mathcal{K}_{H_1 (M)}$ as 
local coefficient systems on $\Sigma_{g,1}$ and $M$ respectively. 
By an argument using covering spaces, we have the following. 
We refer to \cite[Proposition 2.10]{cot} and 
\cite[Proposition 2.1]{klw} for the proof. 
\begin{lemma}\label{zsakasa:lem1}
$i_{\pm}: 
H_\ast (\Sigma_{g,1},p ;i_{\pm}^\ast \mathcal{K}_{H_1 (M)}) \to 
H_\ast (M,p ;\mathcal{K}_{H_1 (M)})$ are isomorphisms of 
right $\mathcal{K}_{H_1 (M)}$-vector spaces. 
\end{lemma}
\noindent
This lemma plays an important role in defining our invariants below. 

\bigskip
\noindent
{\bf (I) Magnus representation} \ 

By using the spine $R_{2g}$ taken in the previous section, we 
identify $\pi_1 (\Sigma_{g,1})=\langle \gamma_1, \ldots , \gamma_{2g} \rangle$ 
with a free group $F_{2g}$ of rank $2g$. Since $R_{2g} \subset \Sigma_{g,1}$ 
is a deformation retract, we have
\begin{align*}
H_1 (\Sigma_{g,1},p\,;i_{\pm}^\ast \mathcal{K}_{H_1 (M)}) &\cong 
H_1 (R_{2g},p\,;i_{\pm}^\ast \mathcal{K}_{H_1 (M)}) \\
&= C_1 (\widetilde{R_{2g}}) \otimes_{F_{2g}} 
i_{\pm}^\ast \mathcal{K}_{H_1 (M)} 
\cong \mathcal{K}_{H_1 (M)}^{2g}
\end{align*}
\noindent
with a basis $\{ \widetilde{\gamma_1} \otimes 1, \ldots , 
\widetilde{\gamma_{2g}} \otimes 1\} 
\subset C_1 (\widetilde{R_{2g}}) \otimes_{F_{2g}} 
i_{\pm}^\ast \mathcal{K}_{H_1 (M)}$ 
as a right free $\mathcal{K}_{H_1 (M)}$-module, where 
$\widetilde{\gamma_i}$ is a lift 
of $\gamma_i$ on the universal covering $\widetilde{R_{2g}}$. 
We denote by $\mathcal{K}_{H_1 (M)}^{2g}$ the space of 
column vectors with $2g$ entries in $\mathcal{K}_{H_1 (M)}$. 

\begin{definition}\label{def:Magnus}
$(1)$ For $M=(M,i_+,i_-) \in \mathcal{C}_{g,1}$, we 
denote by $r'(M) \in \mathrm{GL}(2g,\mathcal{K}_{H_1 (M)})$ 
the representation matrix of 
the right $\mathcal{K}_{H_1 (M)}$-isomorphism
\[\mathcal{K}_{H_1 (M)}^{2g} \cong 
H_1 (\Sigma_{g,1},p\,;i_-^\ast \mathcal{K}_{H_1 (M)}) 
\xrightarrow[i_+^{-1} \circ i_-]{\cong} 
H_1 (\Sigma_{g,1},p\,;i_+^\ast \mathcal{K}_{H_1 (M)}) 
\cong \mathcal{K}_{H_1 (M)}^{2g}\]
$(2)$ The {\it Magnus representation\/} for $\mathcal{C}_{g,1}$ 
is the map $r :\mathcal{C}_{g,1} \to \mathrm{GL}(2g,\mathcal{K}_{H})$ 
which assigns to $M=(M,i_+,i_-) \in \mathcal{C}_{g,1}$ the 
matrix $r(M):={}^{i_+^{-1}} r'(M)$ obtained from $r'(M)$ by applying 
$i_+^{-1}$ to each entry.
\end{definition}
\noindent
We call $r(M)$ the {\it Magnus matrix} for $M$. 
The map $r$ has the following properties: 

\begin{theorem}[\cite{sakasai08, sa3}]\label{thm:magnus_properties} 
$(1)$ $($Crossed homomorphism\/$)$ \ For $M_1, M_2 \in \mathcal{C}_{g,1}$, 
we have 
\[r(M_1 \cdot M_2) = r(M_1) \cdot {}^{\sigma (M_1)} r(M_2).\]
In particular, the restriction of $r$ to $\mathcal{IC}_{g,1}$ 
is a homomorphism. 

\noindent
$(2)$ $($Symplecticity\/$)$ \ 
For any $M \in \mathcal{C}_{g,1}$, we have the equality 
\[\overline{r (M)^T} \ \widetilde{J} \ r (M) = 
{}^{\sigma (M)} \widetilde{J},\]
where $\overline{r (M)^T}$ is obtained from $r (M)$ by taking 
the transpose and applying the involution induced from the map 
$(H \ni x \mapsto x^{-1} \in H)$ to each entry, and 
$\widetilde{J} \in \mathrm{GL} (2g, \mathbb{Z} [H])$ 
is the matrix which appeared in Papakyriakopoulos' paper \cite{papa}. 
(The matrix $J$ is mapped to the usual symplectic matrix by applying 
the trivializer $\mathbb{Z}[H] \to \mathbb{Z}$ 
to each entry.)

\noindent
$(3)$ $($Homology cobordism invariance\/$)$ \ The map $r:\mathcal{C}_{g,1} \to 
\mathrm{GL}(2g,\mathcal{K}_{H})$ induces a crossed homomorphism 
$r:\mathcal{H}_{g,1} \to \mathrm{GL}(2g,\mathcal{K}_{H})$ and its restriction 
to $\mathcal{IH}_{g,1}$ is a homomorphism. 
\end{theorem} 

\bigskip
\noindent
{\bf (II) $H$-torsion} 

Since the relative complex 
$C_\ast (M,i_+(\Sigma_{g,1}); \mathcal{K}_{H_1 (M)})$ 
obtained from any smooth triangulation of $(M,i_+(\Sigma_{g,1}))$ 
is acyclic by Lemma \ref{zsakasa:lem1}, 
we can define its Reidemeister torsion 
\[\tau(C_\ast (M,i_+(\Sigma_{g,1});\mathcal{K}_{H_1 (M)})) \in 
\mathcal{K}_{H_1 (M)}^\times/(\pm H_1 (M)),\]
where 
$\mathcal{K}_{H_1 (M)}^\times := \mathcal{K}_{H_1 (M)}-\{0\}$ is 
the unit group of $\mathcal{K}_{H_1 (M)}$. 
We refer to Milnor \cite{mi} and Turaev \cite{tu2} 
for generalities of Reidemeister torsions. 

\begin{definition}
The $H$-{\it torsion\/} $\tau (H)$ of a homology cylinder 
$M=(M,i_+,i_-) \in \mathcal{C}_{g,1}$ 
is defined by 
\[\tau (M):={}^{i_+^{-1}}\tau(C_\ast (M,i_+(\Sigma_{g,1});
\mathcal{K}_{H_1 (M)}))
\in \mathcal{K}_{H}^\times/(\pm H),\]
where $\mathcal{K}_{H}^\times=\mathcal{K}_{H}-\{0\}$ is 
the unit group of $\mathcal{K}_{H}$. 
\end{definition} 
\noindent
The map $\tau: \mathcal{C}_{g,1} \to \mathcal{K}_{H}^\times/(\pm H)$ has 
the following properties: 

\begin{theorem}\label{thm:torsion_properties} 
$(1)$ $($Crossed homomorphism \cite{sakasai08}$)$ \ 
For $M_1, M_2 \in \mathcal{C}_{g,1}$, 
we have 
\[\tau (M_1 \cdot M_2) = \tau (M_1) \cdot {}^{\sigma (M_1)} \tau (M_2).\]
In particular, the restriction of $\tau$ to $\mathcal{IC}_{g,1}$ 
is a homomorphism. 

\noindent
$(2)$ $(${\rm Cha-Friedl-Kim} \cite[Theorem 3.10]{cfk}, 
{\rm Turaev} \cite[Theorem 1.11.2]{turaev_knot}$)$ \ 
If $M,N \in \mathcal{C}_{g,1}$ are 
homology cobordant, then there exists 
$q \in \mathcal{K}_H^\times$ such that 
\[\tau (M)= \tau (N) \cdot q \cdot \overline{q} 
\in \mathcal{K}_H^\times/(\pm H).\]
\end{theorem} 
\noindent
Note that the restriction of $\tau$ to $\mathcal{M}_{g,1}$ is 
trivial since $\Sigma_{g,1} \times [0,1]$ is simple homotopy equivalent to 
$\Sigma_{g,1} \times \{1\}$.

Explicit formulas for  
$r(M)$ and $\tau (M)$ are given in \cite[Section 4]{gs08}, 
which are based on the formulas for the corresponding invariants of string links 
by Kirk-Livingston-Wang \cite{klw}. 
An {\it admissible presentation} of $\pi_1 (M)$ is defined to be 
a presentation of the form 
\[\langle i_- (\gamma_1),\ldots,i_- (\gamma_{2g}), 
z_1 ,\ldots, z_l, 
i_+ (\gamma_1),\ldots,i_+ (\gamma_{2g}) \mid 
r_1, \ldots, r_{2g+l}
\rangle\]
for some integer $l \ge 0$. 
That is, it is a finite presentation with deficiency $2g$ 
whose generating set 
includes $i_- (\gamma_1),\ldots,i_- (\gamma_{2g}), 
i_+ (\gamma_1),\ldots,i_+ (\gamma_{2g})$ and is ordered as above. 
Such a presentation always exists. For any admissible presentation, 
we define $2g \times (2g+l)$, $l \times (2g+l)$ and 
$2g \times (2g+l)$ matrices $A,B,C$ by 
\[A=\sideset{^{i_+^{-1}}\!}{_{\begin{subarray}{c}
{}1 \le i \le 2g\\
1 \le j \le 2g+l
\end{subarray}}}
{\Tmatrix{
\overline{\left(
\displaystyle\frac{\partial r_j}{\partial i_-(\gamma_i)}
\right)}}}, \ 
B=\sideset{^{i_+^{-1}}\!}{_{\begin{subarray}{c}
{}1 \le i \le l\\
1 \le j \le 2g+l
\end{subarray}}}
{\Tmatrix{
\overline{\left(
\displaystyle\frac{\partial r_j}{\partial z_i}
\right)}}}, \ 
C=\sideset{^{i_+^{-1}}\!}{_{\begin{subarray}{c}
{}1 \le i \le 2g\\
1 \le j \le 2g+l
\end{subarray}}}
{\Tmatrix{
\overline{\left(
\displaystyle\frac{\partial r_j}{\partial i_+(\gamma_i)}
\right)}}}\]
over $\Z [H] \subset \mathcal{K}_H$. 
\begin{proposition}[{\cite[Propositions 4.5, 4.6]{gs08}}]
\label{prop:MagnusFormula}
For any homology cylinder $M=(M,i_+, i_-) \in \Cg$, 
the square matrix $\begin{pmatrix} A \\ B \end{pmatrix}$ 
is invertible over $\mathcal{K}_H$ and we have 
\begin{align*}
r (M) &= 
-C \begin{pmatrix} A \\ B \end{pmatrix}^{-1} \!
\begin{pmatrix} I_{2g} \\ 0_{(l,2g)}\end{pmatrix} 
\in \mathrm{GL}(2g,\mathcal{K}_H),\\
\tau (M) &= \det \begin{pmatrix} A \\ B \end{pmatrix} \in 
\mathcal{K}_H^\times/(\pm H).
\end{align*}
\end{proposition}

\begin{example}[{\cite[Examples 4.4, 6.2]{sakasai08}}]\label{ex:string_comp}
Let $L$ be the string link of $2$ strings depicted in 
Figure \ref{fig:ex4_4}. 
We can construct a homology cylinder 
$(M_L,i_+,i_-) \in \mathcal{C}_{2,1}$ as 
mentioned in Example \ref{ex:stringlink}.

\begin{figure}[htbp]
\begin{center}
\includegraphics{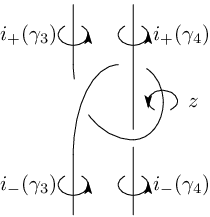}
\end{center}
\caption{The string link $L$}
\label{fig:ex4_4}
\end{figure}

An admissible presentation of $\pi_1 (M_L)$ is given by 
\[{\small \left\langle
\begin{array}{c|c}
\begin{array}{c}
i_-(\gamma_1),\ldots,i_-(\gamma_4)\\
z\\
i_+(\gamma_1),\ldots,i_+(\gamma_4)\\
\end{array} &
\begin{array}{l}
i_+(\gamma_1) i_-(\gamma_3)^{-1} i_+(\gamma_4) i_-(\gamma_1)^{-1},\\
{[i_+(\gamma_1),i_+(\gamma_3)]} i_+(\gamma_2) z i_-(\gamma_2)^{-1} 
{[i_-(\gamma_3),i_-(\gamma_1)]},\\
i_+(\gamma_4) i_-(\gamma_3) i_+(\gamma_4)^{-1} z^{-1},\\
i_-(\gamma_3) i_+(\gamma_3)^{-1} i_-(\gamma_3)^{-1} z, \ 
i_-(\gamma_4) z^{-1} i_+(\gamma_4)^{-1} z
\end{array}
\end{array}\right\rangle}.\] 
By using Proposition \ref{prop:MagnusFormula}, we have 
\begin{align*}
r(M_L) &= 
\left(\begin{array}{cccc}
1&0&0&0\\
&&&\\
0&1&0&0
\\
&&&\\
\frac{-\gamma_1^{-1}}{\gamma_3^{-1}+\gamma_4^{-1}-1}
&\frac{\gamma_2^{-1} \gamma_3^{-1} \gamma_4^{-1} 
-\gamma_4^{-1}+1}{\gamma_3^{-1}+\gamma_4^{-1}-1} &
\frac{\gamma_3^{-1}}{\gamma_3^{-1}+\gamma_4^{-1}-1}
&\frac{\gamma_4^{-1}(\gamma_4^{-1}-1)}{\gamma_3^{-1}+\gamma_4^{-1}-1}\\
&&&\\
\frac{\gamma_1^{-1}\gamma_3\gamma_4^{-1}}{\gamma_3^{-1}+\gamma_4^{-1}-1}
& \frac{(1-\gamma_3^{-1})(\gamma_2^{-1} \gamma_3^{-1} 
-\gamma_2^{-1}-1)}{\gamma_3^{-1}+\gamma_4^{-1}-1} 
&\frac{\gamma_3^{-1}-1}{\gamma_3^{-1}+\gamma_4^{-1}-1}
&\frac{-\gamma_3^{-1}\gamma_4^{-1}+\gamma_3^{-1}
+2\gamma_4^{-1}-1}{\gamma_3^{-1}+\gamma_4^{-1}-1}
\end{array}\right),\\
\tau (M_L) &= -1 + \gamma_3 -\gamma_3 \gamma_4^{-1}=
-\gamma_3(\gamma_3^{-1}+\gamma_4^{-1}-1).
\end{align*}
\noindent
Note that 
\[\det (r(M_L))= \gamma_3^{-1} \gamma_4^{-1} 
\frac{\gamma_3 + \gamma_4-1}{\gamma_3^{-1}+\gamma_4^{-1}-1}.\]
\end{example}

%----------Abelian quotients----------2011/8/31----

\section{Abelian quotients}\label{sec:abelian}
In this section, we discuss abelian quotients of 
$\mathcal{C}_{g,1}$ and $\mathcal{H}_{g,1}$ by  
comparing them to the corresponding result for $\mathcal{M}_{g,1}$. 
First, as commented in \cite{gs09}, 
we point out that $\mathcal{C}_{g,1}$ has the monoid $\theta_\mathbb{Z}^3$ 
of homology 3-spheres as a big abelian quotient. 
In fact, we have 
a {\it forgetful} homomorphism $F:\mathcal{C}_{g,1} 
\to \theta_\mathbb{Z}^3$ defined by 
$F(M,i_+,i_-) = S^3 \sharp X_1 \sharp X_2 \sharp 
\cdots \sharp X_n$ for the prime decomposition 
$M=M_0 \sharp X_1 \sharp X_2 \sharp 
\cdots \sharp X_n$ of $M$ where $M_0$ is the unique factor having 
non-empty boundary and $X_i \in \theta_\mathbb{Z}^3$ $(1 \le i \le n)$. 
The map $F$ owes its well-definedness to 
the uniqueness of the prime decomposition of 3-manifolds and 
it is a monoid epimorphism. 

The underlying 3-manifolds of homology cylinders obtained 
from $\mathcal{M}_{g,1}$ are all 
$\Sigma_{g,1} \times [0,1]$ and, in particular, irreducible. 
Therefore it seems more reasonable to compare $\mathcal{M}_{g,1}$ with 
the submonoid $\mathcal{C}_{g,1}^\mathrm{irr}$ of 
$\mathcal{C}_{g,1}$ consisting of all $(M,i_+,i_-)$ with $M$ irreducible. 

In contrast with the fact that $\mathcal{M}_{g,1}$ is a perfect group for 
$g \ge 3$ (see Harer \cite{harer}), 
many infinitely generated abelian quotients for 
monoids and homology cobordism groups of irreducible 
homology cylinders have been found until now. 
For example, we have the following results:

\begin{itemize}
\item \ In \cite[Corollary 6.16]{sakasai08}, 
we showed that the submonoids $\mathcal{C}_{g,1}^\mathrm{irr} \cap 
\mathcal{IC}_{g,1}$ 
and $\mathrm{Ker}\, (\mathcal{C}_{g,1}^\mathrm{irr} \to \mathcal{H}_{g,1})$ 
have abelian quotients isomorphic to $(\mathbb{Z}_{\ge 0})^\infty$. 
The proof uses the $H$-torsion $\tau$ and 
its non-commutative generalization. 

\item \ Morita \cite[Corollary 5.2]{morita_GD} used what is called 
the trace maps 
to show that the group $\mathcal{IH}_{g,1}$, which coincides with 
the quotient of $\mathcal{C}_{g,1}^\mathrm{irr} \cap 
\mathcal{IC}_{g,1}$ by homology cobordisms, has an abelian quotient 
isomorphic to $\mathbb{Z}^\infty$. 

\item \ In a joint work with Goda in \cite[Theorem 2.6]{gs09}, we showed that 
$\mathcal{C}_{g,1}^\mathrm{irr}$ has an abelian quotient 
isomorphic to $(\mathbb{Z}_{\ge 0})^\infty$ by using 
sutured Floer homology (a variant of Heegaard Floer homology). However, 
the projection map to 
this abelian quotient does not factor through $\mathcal{H}_{g,1}$. 
\end{itemize}

By taking into account the similarity between 
the two groups $\mathcal{M}_{g,1}$ and $\mathcal{H}_{g,1}$, 
it had been conjectured that 
$\mathcal{H}_{g,1}$ was perfect. 
However, Cha-Friedl-Kim \cite{cfk} found 
a method for extracting homology cobordism invariants 
of homology cylinders from 
the $H$-torsion $\tau : \mathcal{C}_{g,1} \to 
\mathcal{K}_H^\times/(\pm H)$, which is a crossed homomorphism, 
as follows. 

First they consider the subgroup $A \subset \mathcal{K}_H^\times$ 
defined by
\[A:=\{f^{-1} \cdot \varphi (f) \mid f \in \mathcal{K}_H^\times, \ 
\varphi \in \mathrm{Sp} (2g,\mathbb{Z})\},\]
by which we can obtain a {\it homomorphism} 
\[\tau : 
\mathcal{C}_{g,1} \longrightarrow \mathcal{K}_H^\times/(\pm H \cdot A).\]
Note that $f=\overline{f}$ holds in 
$\mathcal{K}_H^\times/(\pm H \cdot A)$ since $-I_{2g} \in 
\mathrm{Sp}(2g,\mathbb{Z})$. 
Second, they use the equality mentioned in 
Theorem \ref{thm:torsion_properties} (2). Namely, if we put  
\[N:=\{f \cdot \overline{f} \mid f \in \mathcal{K}_H^\times \},\]
then we obtain a homomorphism 
\[\widetilde{\tau}: 
\mathcal{H}_{g,1} \longrightarrow 
\mathcal{K}_H^\times/(\pm H \cdot A \cdot N).\]
Note that 
$f^2=f \overline{f}=1$ holds for any $f \in \mathcal{K}_H^\times/(\pm H \cdot A \cdot N)$. 

The structure 
of $\mathcal{K}_H^\times/(\pm H \cdot A \cdot N)$ is given as follows. 
Recall that $\mathcal{K}_H = \mathbb{Z} [H](\mathbb{Z} [H] -\{0\})^{-1}$. 
The ring $\mathbb{Z} [H]$ is a Laurent polynomial ring of $2g$ variables and 
it is a unique factorization domain. 
Thus every Laurent polynomial $f$ is factorized 
into irreducible polynomials 
uniquely up to multiplication by a unit in $\mathbb{Z} [H]$. 
In particular, for every irreducible polynomial $\lambda \in \mathbb{Z} [H]$, 
we can count the exponent of $\lambda$ in the factorization of $f$. 
This counting naturally extends to that for elements in 
$\mathcal{K}_H^\times$ by using negative numbers for denominators. 
Under the identification by 
$\pm H \cdot A \cdot N$, an element in 
$\mathcal{K}_H^\times/(\pm H \cdot A \cdot N)$ is determined by 
the exponents of all $\mathrm{Sp}(2g,\mathbb{Z})$-orbits 
of irreducible polynomials 
(up to multiplication by a unit in $\mathbb{Z} [H]$) modulo $2$. 
Note that the action of $\mathrm{Sp}(2g,\mathbb{Z})$ keeps the irreducibility of 
a polynomial in $\mathbb{Z} [H]$ and 
the number of its monomials unchanged. 
Therefore $\mathcal{K}_H^\times/(\pm H \cdot A \cdot N)$ is 
isomorphic to $(\mathbb{Z}/2\mathbb{Z})^\infty$. 
Finally by using infinitely many 
$(\mathbb{Z}/2\mathbb{Z})$-torsion elements of the knot concordance group, 
they show the following:
\begin{theorem}[Cha-Friedl-Kim \cite{cfk}]\label{thm:cfk}
The image of the homomorphism 
\[\widetilde{\tau} : 
\mathcal{H}_{g,1} \longrightarrow 
\mathcal{K}_H^\times/(\pm H \cdot A \cdot N)\]
is isomorphic to $(\mathbb{Z}/2\mathbb{Z})^\infty$ and it splits. 
\end{theorem}
\begin{remark}
In \cite{cfk}, Cha-Friedl-Kim further applied the above method to 
homology cylinders over a compact oriented surface $\Sigma_{g,n}$ 
of genus $g$ with $n \ge 2$ boundaries and showed that the abelianization of 
the homology cobordism group $\mathcal{H}(\Sigma_{g,n})$ 
has infinite rank.  
\end{remark}

Now we try to investigate abelian quotients of $\mathcal{H}_{g,1}$ by 
using the Magnus representation $r$. 
It looks easier to extract information of $\mathcal{H}_{g,1}$ 
from the representation $r$ together with Cha-Friedl-Kim's idea, 
since $r$ itself is an 
homology cobordism invariant as mentioned in 
Theorem \ref{thm:magnus_properties} (3). 
Consider two maps
\begin{align*}
\widehat{r} &: \mathcal{H}_{g,1} 
\stackrel{r}{\longrightarrow} \mathrm{GL}(2g,\mathcal{K}_H)
\xrightarrow{\det} \mathcal{K}_H^\times 
\longrightarrow \mathcal{K}_H^\times/(\pm H),\\
\widetilde{r} &: \mathcal{H}_{g,1} 
\stackrel{\widehat{r}}{\longrightarrow} 
\mathcal{K}_H^\times/(\pm H)
\longrightarrow \mathcal{K}_H^\times/(\pm H \cdot A).
\end{align*}
\noindent
While $\widehat{r}$ is a crossed homomorphism, 
its restriction to $\mathcal{IH}_{g,1}$ and 
$\widetilde{r}$ are homomorphisms. 
Note that both $\mathcal{K}_H^\times/(\pm H)$ and 
$\mathcal{K}_H^\times/(\pm H \cdot A)$ are isomorphic to 
$\mathbb{Z}^\infty$. 

\begin{theorem}
$(1)$ \  For $(M,i_+,i_-) \in \mathcal{C}_{g,1}$, 
the equality
\[\widehat{r}(M) = 
\overline{\tau (M)} \cdot 
(\tau (M))^{-1} \ \ 
\in \mathcal{K}_H^\times/(\pm H)\]
holds. 

\noindent
$(2)$ \  For $g \ge 1$, the homomorphism 
$\widetilde{r}: \mathcal{H}_{g,1} \to 
\mathcal{K}_H^\times/(\pm H \cdot A)$ is trivial. 

\noindent
$(3)$ \ For $g \ge 2$, the image of 
the homomorphism $\widehat{r}|_{ \mathcal{IH}_{g,1}}: \mathcal{IH}_{g,1} \to 
\mathcal{K}_H^\times/(\pm H)$ is isomorphic to $\mathbb{Z}^\infty$. 
\end{theorem}
\begin{proof}
$(1)$ Let 
\[\langle i_- (\gamma_1),\ldots,i_- (\gamma_{2g}), 
z_1 ,\ldots, z_l, 
i_+ (\gamma_1),\ldots,i_+ (\gamma_{2g}) \mid 
r_1, \ldots, r_{2g+l}
\rangle\]
be an admissible presentation of $\pi_1 (M)$. We calculate 
the matrices $A$, $B$, $C$ as in the previous section. 
By the formula in Proposition \ref{prop:MagnusFormula}, we have an 
equality 
\[\begin{pmatrix}
r(M) & Z \end{pmatrix} = -C \begin{pmatrix} A \\ B \end{pmatrix}^{-1}\]
for some $2g \times l$ matrix $Z$. It follows that 
$r(M) A = -ZB -C$. By taking the determinant of
\[\begin{pmatrix}
r(M) & 0_{(2g,l)} \\ 0_{(l,2g)} & I_l
\end{pmatrix}
\begin{pmatrix} A \\ B \end{pmatrix}=
\begin{pmatrix} r(M)A \\ B \end{pmatrix}=
\begin{pmatrix} -ZB-C \\ B \end{pmatrix},
\]
we have the equality
\[\det (r(M)) \det \begin{pmatrix} A \\ B \end{pmatrix}=
\det \begin{pmatrix} -ZB-C \\ B \end{pmatrix} = 
\det \begin{pmatrix} -C \\ B \end{pmatrix}=
\det \begin{pmatrix} C \\ B \end{pmatrix}.\]
Again by Proposition \ref{prop:MagnusFormula}, 
we have $\det \begin{pmatrix} A \\ B \end{pmatrix}=\tau(M)$ and 
it is easy to see that 
\[\det \begin{pmatrix} C \\ B \end{pmatrix}={}^{\sigma(M)} \tau(M^{-1}),\]
where $M^{-1}=(-M,i_-,i_+) \in \mathcal{C}_{g,1}$. 
Recall that $\tau(M)$ is the pullback of the torsion of the complex 
$C_\ast (M,i_+(\Sigma_{g,1});\mathcal{K}_{H_1 (M)})$ by $i_+$, and 
$\tau (M^{-1})$ is that of 
$C_\ast (M,i_-(\Sigma_{g,1});\mathcal{K}_{H_1 (M)})$ by $i_-$. 
These complexes are related by the Poincar\'e duality. 
By the duality of Reidemeister torsions 
(see Milnor \cite{mi} or Turaev\cite{tu2}), we have 
\begin{align*}
{}^{i_-} \tau(M^{-1})&=
\tau (C_\ast (M,i_-(\Sigma_{g,1});\mathcal{K}_{H_1 (M)})) \\
&= \overline{\tau(C_\ast (M,i_+(\Sigma_{g,1});\mathcal{K}_{H_1 (M)}))}
=\overline{{}^{i_+} \tau(M)} \in \mathcal{K}_{H_1 (M)}/(\pm H_1 (M)).
\end{align*}
\noindent
Therefore we have ${}^{\sigma(M)} \tau(M^{-1})=\overline{\tau(M)}$. 
Our claim follows from this. 

$(2)$  As mentioned above, the action of $\mathrm{Sp}(2g,\mathbb{Z})$ implies that 
$f=\overline{f}$ for any $f \in \mathcal{K}_H^\times/(\pm H \cdot A)$. 
Then the claim immediately follows from $(1)$. 
(We may also use the symplecticity of the image of $r$ 
mentioned in Theorem \ref{thm:magnus_properties} (2).) 

$(3)$ We use the homology cylinder $M_L \in \mathcal{C}_{2,1}$ 
in Example \ref{ex:string_comp}. 
While $M_L \notin \mathcal{IC}_{2,1}$, we 
can adjust it by some $g_1 \in \mathcal{M}_{2,1}$ so that 
$M_L \cdot g_1 \in \mathcal{IC}_{2,1}$. Since 
$\widehat{r}$ is trivial on $\mathcal{M}_{2,1}$, 
we have 
\[\widehat{r}  (M_L \cdot g_1) = 
\widehat{r} (M_L) =
\frac{\gamma_3 + \gamma_4-1}{\gamma_3^{-1}+\gamma_4^{-1}-1}
\in \mathcal{K}_H^\times/(\pm H).\] 
Take $f \in \mathcal{M}_{2,1}$ such that 
$\sigma (f) \in \mathrm{Sp}(4,\mathbb{Z})$ maps 
\[\gamma_1 \longmapsto \gamma_1 + \gamma_4, 
\quad \gamma_2 \longmapsto \gamma_2, 
\quad \gamma_3 \longmapsto \gamma_2 + \gamma_3, \quad 
\gamma_4 \longmapsto \gamma_4.\] 
Consider $f^m \cdot M_L \in \mathcal{C}_{2,1}$ 
and adjust it by 
some $g_m \in \mathcal{M}_{2,1}$ so that $f^m \cdot M_L \cdot g_m \in 
\mathcal{IC}_{2,1}$. 
Then we have 
\[\widehat{r} (f^m \cdot M_L \cdot g_m) = 
{}^{\sigma (f^m)} \widehat{r} (M_L)=
\frac{\gamma_2^m \gamma_3 + \gamma_4-1}
{\gamma_2^{-m}\gamma_3^{-1}+\gamma_4^{-1}-1} 
\in \mathcal{K}_H^\times/(\pm H).\] 
Since $\gamma_2^m \gamma_3 + \gamma_4-1$ is a degree $1$ polynomial 
with respect 
to the variable $\gamma_3$ and the coefficient of $\gamma_3$ is 
a monomial, we see that it is irreducible. 
By applying the involution, 
the irreducibility of $\gamma_2^{-m}\gamma_3^{-1}+\gamma_4^{-1}-1$ follows. 
It is easily checked that 
\begin{align*}
&\gamma_2^m \gamma_3 + \gamma_4-1 \neq 
\gamma_2^{-m}\gamma_3^{-1}+\gamma_4^{-1}-1 \\
&\gamma_2^m \gamma_3 + \gamma_4-1 \neq \gamma_2^k \gamma_3 + \gamma_4-1 
\quad (m \neq k)
\end{align*}
\noindent
as elements of $\mathcal{K}_H^\times/(\pm H)$ by considering 
the ratios among monomials, which are invariant under the multiplication of any element of $\pm H$. 
Therefore we conclude that 
the values 
\[\left\{\displaystyle\frac{\gamma_2^m \gamma_3 + \gamma_4-1}
{\gamma_2^{-m}\gamma_3^{-1}+\gamma_4^{-1}-1}\right\}_{m=0}^\infty\]
generate an infinitely generated subgroup of 
$\mathcal{K}_H^\times/(\pm H)$. This completes the proof when $g=2$. 
We can use the above computation for $g \ge 3$ by a stabilization.
\end{proof}
\noindent
From the above theorem, we observe that it seems not easy 
to find new abelian quotients of $\mathcal{H}_{g,1}$ by using 
the Magnus representation.

%----------Generalization to higher-dimensional cases--------------

\section{Generalization to higher-dimensional cases}\label{sec:higher}

In the remaining sections, we apply the argument in the previous section 
to higher dimensional cases and see that the determinant of the Magnus 
representation works well for them. 

For $k \ge 2$ and $n \ge 1$, we put 
\[X_n^k := \mathop{\#}_{n} (S^1 \times S^{k-1}).\]
The manifold $X_n^k$ may be regarded as a generalization 
of a closed surface since $X_n^2 = \Sigma_{n,0}$. 

Suppose $k \ge 3$. Then we have $\pi_1 (X_n^k) \cong 
\pi_1 (X_n^k - \mathrm{Int}\,D^k) \cong  F_n$, 
where $\mathrm{Int}\,D^k$ is an open $k$-ball. 
We choose a base point $p$ of $X_n^k - \mathrm{Int}\,D^k$ from 
the boundary and 
take an ordered basis $\{ \gamma_1, \gamma_2, \ldots, \gamma_n \}$ 
of $F_n$ (and $H_1:= H_1(F_n) \cong \mathbb{Z}^n$). 

Similarly to Lemma \ref{zsakasa:lem1}, we can check that 
\[i_{\pm}: H_\ast (X_n^k - \mathrm{Int}\,D^k, p \,;
i_{\pm}^\ast \mathcal{K}_{H_1 (M)}) \to 
H_\ast (M,p \,;\mathcal{K}_{H_1 (M)})\] 
are isomorphisms of 
right $\mathcal{K}_{H_1 (M)}$-vector spaces 
for any homology cylinder $(M,i_+,i_-)$ 
over $X_n^k - \mathrm{Int}\,D^k$, where 
$\mathcal{K}_{H_1}=\mathbb{Z}[H_1](\mathbb{Z}[H_1]-\{0\})^{-1}$. 
Hence we can define the Magnus representation 
\[r: \mathcal{C}(X_n^k - \mathrm{Int}\,D^k) 
\longrightarrow \mathrm{GL}(n,\mathcal{K}_{H_1})\]
by the same procedure as before. 
The map $r$ is a crossed homomorphism and induces 
$r:\mathcal{H} (X_n^k - \mathrm{Int}\,D^k) \to 
\mathrm{GL}(n,\mathcal{K}_{H_1})$. 
Consider the composition 
\[\widetilde{r}: \mathcal{H} (X_n^k - \mathrm{Int}\,D^k) 
\stackrel{r}{\longrightarrow} 
\mathrm{GL}(n,\mathcal{K}_{H_1}) 
\stackrel{\det}{\longrightarrow} 
\mathcal{K}_{H_1}^\times 
\longrightarrow 
\mathcal{K}_{H_1}^\times/(\pm H_1 \cdot A') \cong \mathbb{Z}^\infty,\]
where $A':=\{f^{-1} \cdot \varphi (f) \mid 
f \in \mathcal{K}_{H_1}^\times, \ 
\varphi \in \Aut (H_1) \}$, 
which gives a homomorphism. 

Now we mention 
the main result in the remaining sections. 
Note that we have a surjective homomorphism 
$\mathcal{H} (X_n^k -\mathrm{Int}\, D^k) \twoheadrightarrow 
\mathcal{H} (X_n^k)$ by gluing 
a small trivial cylinder along the boundary, which corresponds to 
capping the boundary of $X_n^k-\mathrm{Int}\, D^k$ by a $k$-ball $D^k$. 

\begin{theorem}\label{thm:acyMag}
For any $k \ge 3$ and $n \ge 2$, we have:
\begin{itemize}
\item[$(1)$] The image of the homomorphism $\widetilde{r}$ 
is an infinitely generated subgroup of $\mathbb{Z}^\infty$. In particular, 
$H_1 (\mathcal{H} (X_n^k -\mathrm{Int}\, D^k) )$ has 
infinite rank. 

\item[$(2)$] The homomorphism $\widetilde{r}$ factors through 
$\mathcal{H} (X_n^k)$. Therefore 
$H_1 (\mathcal{H} (X_n^k) )$ has infinite rank. 
\end{itemize}
\end{theorem}
For the proof, which occupies Sections \ref{sec:acyclic} and 
\ref{sec:magnus2}, we use the action 
of $\mathcal{H} (X_n^k -\mathrm{Int}\, D^k)$ on the group called the 
{\it acyclic closure} of $F_n$. This action may be regarded as 
a generalization of the action of 
the diffeotopy group $\mathcal{M}(X_n^k -\mathrm{Int}\, D^k)$ on 
$F_n=\pi_1 (X_n^k -\mathrm{Int}\, D^k)$. 
Recall that the Magnus representation was originally defined for 
automorphisms of $F_n$ by using the Fox derivatives. 
The Magnus representation for homology cylinders we have seen is 
an extension of this representation by using 
twisted homology. In the following sections, we describe 
an equivalent definition of our Magnus representation by using 
the extended Fox derivatives first given by Le Dimet \cite{ld}.

%----------The acyclic closure of a group--------------

\section{The acyclic closure of a group}\label{sec:acyclic}

The notion of the acyclic closure (or HE-closure in \cite{le2}) 
of a group was defined as a variation of the algebraic closure of a group 
by Levine \cite{le1,le2}.  
We summarize here the 
definition and fundamental properties. 
We also refer to Hillman's book \cite{hi} and Cha's paper \cite{cha}. 

\begin{definition}
Let $G$ be a group, and let $F_m=\langle x_1, x_2, \ldots, x_m 
\rangle$ be a free group of rank $m$. \\
(i) $w=w(x_1,x_2, \ldots, x_m) \in G \ast F_m$, 
a word in $x_1, x_2, \ldots, x_m$ and elements of $G$, 
is said to be  {\it acyclic} if 
\[w \in \Ker \left( G \ast F_m \xrightarrow{\mathrm{proj}} F_m 
\longrightarrow H_1 (F_m) \right).\]
(ii) Consider the following ``equation'' with variables 
$x_1,x_2,\ldots,x_m$:
\[
\left\{\begin{array}{ccl}
x_1 & = & w_1(x_1, x_2, \ldots, x_m)\\
x_2 & = & w_2(x_1, x_2, \ldots, x_m)\\
& \vdots & \\
x_m & = & w_m(x_1, x_2, \ldots, x_m)
\end{array}\right. .
\]
When all words $w_1,w_2,\ldots,w_m \in G \ast F_m$ are 
acyclic, we call such an equation an {\it acyclic system} 
over $G$.\\
(iii) A group $G$ is said to be {\it acyclically closed} 
(AC, for short) if 
every acyclic system over $G$ with $m$ variables has 
a unique ``solution'' in $G$ for any $m \ge 0$, where 
a ``solution'' means a homomorphism $\varphi$ that makes 
the diagram 

\medskip
\hspace{100pt}
\SelectTips{cm}{}
\xymatrix{
G \ar[drr]^{\Id} \ar[d]& & \\
\displaystyle\frac{G \ast F_m}{\langle\!\langle x_1 w_1^{-1},\ldots, 
x_m w_m^{-1}\rangle\!\rangle} \ar[rr]_-\varphi & & G
}

\medskip
\noindent
commutative, where $\langle\!\langle x_1 w_1^{-1},\ldots, 
x_m w_m^{-1}\rangle\!\rangle$ denotes the normal closure in $G \ast F_m$. 
\end{definition}
\noindent
\begin{example}
Let $G$ be an abelian group. For $g_1,g_2,g_3 \in G$, 
consider the equation
\[\left\{\begin{array}{l}
x_1=g_1 x_1 g_2 x_2 x_1^{-1} x_2^{-1}\\
x_2=x_1 g_3 x_1^{-1}
\end{array}\right. ,
\]
which is an acyclic system. Then we have a unique solution 
$x_1=g_1 g_2 , \, x_2=g_3$.
\end{example}
\noindent
As we see from this example, all abelian groups are AC. 
Moreover, it is shown in \cite[Proposition 1]{le1} that AC groups are closed under 
taking intersections, direct products, central extensions, direct limits and inverse limits. 
In particular, all nilpotent groups are AC. 

Let us define the acyclic closure of a group. 
\begin{proposition}[{\cite[Proposition 3]{le1}}]
\label{prop:universality}
For any group $G$, there exists a pair of 
a group $G^{\mathrm{acy}}$ and 
a homomorphism $\iota_G:G \rightarrow G^{\mathrm{acy}}$ 
satisfying the following properties:
\begin{enumerate}
\item $G^{\mathrm{acy}}$ is an AC-group.
\item Let $f:G \rightarrow A$ be a homomorphism 
and suppose that $A$ is an AC-group. 
Then there exists a unique homomorphism 
$f^{\mathrm{acy}}:G^{\mathrm{acy}} \rightarrow A$ 
which satisfies $f^{\mathrm{acy}} \circ \iota_G = f$. 
\end{enumerate}
\noindent
Moreover such a pair is unique up to isomorphism.
\end{proposition}
\begin{definition}
We call $\iota_G$ (or $G^{\mathrm{acy}}$) obtained above 
the {\it acyclic closure} of $G$.
\end{definition}
\noindent
Taking the acyclic closure of a group is functorial, namely, 
for each group homomorphism $f:G_1 \to G_2$, we have 
the induced homomorphism $f^{\mathrm{acy}}:G_1^{\mathrm{acy}} \to 
G_2^{\mathrm{acy}}$ by applying the universal property of 
$G_1^{\mathrm{acy}}$ to the homomorphism 
$\iota_{G_2} \circ f$, and the composition of homomorphisms 
induces that of the corresponding homomorphisms 
on acyclic closures. 

The most important properties of the acyclic closure 
are the following, where a homomorphism is 
said to be 
{\it $2$-connected\/} if it induces an isomorphism on 
the first (group) homology 
and an epimorphism on the second homology. 

\begin{proposition}[{\cite[Proposition 4]{le1}}]\label{prop:2cn}
For any group $G$, the acyclic closure 
$\iota_G:G \rightarrow G^{\mathrm{acy}}$ is $2$-connected.
\end{proposition}
\begin{proposition}[{\cite[Proposition 5]{le1}}]\label{prop:isom}
Let $G_1$ be a finitely generated group and $G_2$ be 
a finitely presentable group. For each 
$2$-connected homomorphism 
$f:G_1 \rightarrow G_2$, the induced homomorphism 
$f^{\mathrm{acy}}:G_1^{\mathrm{acy}} 
\to G_2^{\mathrm{acy}}$ on acyclic closures is an isomorphism. 
\end{proposition}
\noindent
From Proposition \ref{prop:2cn} and Stallings' theorem \cite{st}, 
the nilpotent quotients of a group and 
those of its acyclic closure are isomorphic. 
Note that the homomorphism $\iota_G$ is not necessarily injective: 
consider a perfect group $G$ and the $2$-connected homomorphism 
$G \to \{1\}$. As for a free group $F_n$, 
its residual nilpotency shows 
that $\iota_{F_n}:F_n \to F_n^{\mathrm{acy}}$ is injective. 
We write $\gamma_i \in \Acy_n$ again for the image of $\gamma_i \in F_n$ by 
$\iota_{F_n}:F_n=\langle \gamma_1,\gamma_2,\ldots,\gamma_n \rangle 
\hookrightarrow \Acy_n$. 

Now we return to our discussion on homology cylinders. 
For each homology cylinder $(M,i_+,i_-) \in 
\mathcal{C}(X_n^k-\mathrm{Int}\,D^k)$, the homomorphisms 
$i_\pm : F_n =\pi_1 (X_n^k-\mathrm{Int}\,D^k) \to \pi_1 (M)$ are 
$2$-connected. Hence we have a commutative diagram
\[\begin{CD}
F_n @>i_->> \pi_1 (M) @<i_+<< F_n\\ 
@V\iota_{F_n}VV @V\iota_{\pi_1 (M)}VV @VV\iota_{F_n}V \\
\AC{F_n} @>i_-^{\mathrm{acy}}>\cong> \pi_1 (M)^{\mathrm{acy}} 
@<i_+^{\mathrm{acy}}<\cong< \AC{F_n} 
\end{CD}\]
by Proposition \ref{prop:isom}. 
From this, we obtain a monoid homomorphism 
\[\mathrm{Acy}: \mathcal{C} (X_n^k-\mathrm{Int}\,D^k) 
\longrightarrow \mathrm{Aut}\, (F^{\mathrm{acy}}_n) \]
defined by $\mathrm{Acy} (M,i_+,i_-) = 
(i_+^{\mathrm{acy}})^{-1} \circ i_-^{\mathrm{acy}}$ and 
we can check that it induces a group homomorphism
\[\mathrm{Acy}: \mathcal{H} (X_n^k-\mathrm{Int}\,D^k) 
\longrightarrow \mathrm{Aut}\, (F^{\mathrm{acy}}_n).\] 
For homology cylinders over the closed manifold $X_n^k$, we have similar homomorphisms 
\begin{align*}
&\mathrm{Acy}: \mathcal{C} (X_n^k) 
\longrightarrow \mathrm{Out}\, (F^{\mathrm{acy}}_n),\\
&\mathrm{Acy}: \mathcal{H} (X_n^k) 
\longrightarrow \mathrm{Out}\, (F^{\mathrm{acy}}_n)
\end{align*}
\noindent
using the outer automorphism group 
$\mathrm{Out}\, (F^{\mathrm{acy}}_n):=\mathrm{Aut}\, (F^{\mathrm{acy}}_n)/
\mathrm{Inn}\, (F^{\mathrm{acy}}_n)$ of $\Acy_n$.

\begin{theorem}\label{thm:acysurj}
For any $k \ge 3$ and $n \ge 2$, the homomorphisms 
\begin{align*}
&\mathrm{Acy}: \mathcal{H} (X_n^k -\mathrm{Int}\, D^k) 
\longrightarrow \mathrm{Aut}\, (F^{\mathrm{acy}}_n), \\
&\mathrm{Acy}: \mathcal{H} (X_n^k) 
\longrightarrow \mathrm{Out}\, (F^{\mathrm{acy}}_n)
\end{align*} 
are surjective. 
\end{theorem}
\begin{proof}
It suffices to show the surjectivity of the upper one. 
Given an element $\varphi \in \Aut (\Acy_{n})$, 
we produce a homology cylinder $M=(M,i_+,i_-) \in 
\mathcal{H} (X_n^k -\mathrm{Int}\, D^k)$ satisfying 
$\mathrm{Acy}(M)=\varphi$. 
The construction below is based on the argument in 
Garoufalidis-Levine \cite[Theorem 3]{gl}
and its generalization in \cite[Theorem 6.1]{sa2}. 

First we take two continuous maps 
$f_+, f_-:X_n^k -\mathrm{Int}\, D^k \to K(\Acy_n,1)$ 
corresponding to homomorphisms 
$\iota_{F_n}, \varphi \circ \iota_{F_n}:F_n \to \Acy_n$, 
respectively. 
Since $\partial (X_n^k-\mathrm{Int}\, D^k)=S^{k-1}$ is simply connected, 
we can combine these maps and 
obtain a map $f:=f_+ \cup f_- : 
X_{2n}^k=(X_n^k -\mathrm{Int}\, D^k) \cup (-(X_n^k -\mathrm{Int}\, D^k)) 
\to K(\Acy_n,1)$. 
Let $i_+, i_-: X_n^k -\mathrm{Int}\, D^k \to X_{2n}^k$ be the corresponding 
embeddings onto the domains of $f_+$ and $f_-$. 
The manifold $X_{2n}^k$ is the boundary of 
$M_0:=\mathop{\natural}_{2n} S^1 \times D^k$, 
the boundary connected sum of $2n$ copies of 
$S^1 \times D^k$. 
Since $\pi_1 (M_0) \cong \pi_1 (X_{2n}^k)$, we can extend $f$ to 
the continuous map $\Phi: M_0 \to K(\Acy_n,1)$. 
Note that $H_1 (M_0) \cong \mathbb{Z}^{2n}$ and 
$H_i (M_0) = 0$ for all $i \ge 2$. 

Since $f \circ i_+=\iota_{F_n}: F_n \to \Acy_n$ induces an isomorphism 
on the first homology, we have 
$H_1 (M_0) \cong i_+(H_1 (X_n^k-\mathrm{Int}\, D^k)) 
\oplus \Ker \Phi$. 
To obtain a homology cylinder satisfying 
$\mathrm{Acy}(M)=\varphi$, 
we perform surgery to $M_0$ to kill $\Ker \Phi \cong \mathbb{Z}^n$ 
with keeping $\Phi$ on $X_{2n}^k=\partial M_0$. 
Take an element $\alpha \in H_1 (M_0)$ from a basis of 
$\Ker \Phi$. 

(Case 1) Suppose there exists a representative $C \in \pi_1 (M_0)$ of $\alpha$ 
by a simple closed curve with $\Phi (C)=1 \in \Acy_n$. 
Let $W_1$ be the $(k+2)$-manifold obtained 
from $M_0 \times [0,1]$ by attaching a $2$-handle $S^1 \times D^{k+1}$ 
to $M_0 \times \{ 1 \} \subset \partial (M_0 \times [0,1])$ with any framing. 
We have 
\[\pi_1 (W_1)=\pi_1 (M_0)/ \langle\!\langle C \rangle\!\rangle,\]
where $\langle\!\langle C \rangle\!\rangle$ denotes 
the normal closure of the subgroup generated by $C$. 
The relative chain complex $C_\ast (W_1, M_0 \times [0,1])$ 
associated with the handle decomposition has only one generator 
in degree $2$ and its homology class hits $\alpha \in H_1 (M_0 \times [0,1]) 
\cong H_1 (M_0)$. Therefore 
$H_1 (W_1) \cong H_1 (M_0)/ \langle \alpha \rangle \cong \mathbb{Z}^{2n-1}$ and 
$H_i (W_1) \cong H_i (M_0)$ if $i \neq 1$.  
Since $\Phi (C)=1$, we can extend $\Phi$ to $W_1$. 
We write $\Phi: W_1 \to K(\Acy_n,1)$ again for the extension. 
 
Consider $W_1$ to be a cobordism between $M_0=M_0 \times \{ 0 \}$ 
and a new manifold $M_1$. That is, $\partial W_1=M_0 \cup (-M_1)$. 
By duality, the cobordism $W_1$ is obtained from $M_1 \times [0,1]$ 
by attaching a $k$-handle. Since $k \ge 3$, it follows that 
$H_1 (M_1) \cong H_1 (W_1)$ 
with $\Ker \Phi|_{M_1} \cong \mathbb{Z}^{n-1}$. 

(Case 2) Suppose there does not exist a representative 
$C \in \pi_1 (M_0)$ of $\alpha$ 
by a simple closed curve with $\Phi (C)=1 \in \Acy_n$. 
In this case, we replace $(M_0,i_+,i_-)$ by 
another manifold $(M_{0.5},i_+,i_-)$ 
which is homology bordant to $M_0$ over $K(\Acy_n,1)$ 
and for which we can take a simple closed curve 
representing $\alpha \in H_1 (M_{0.5}) \cong H_1 (M_0)$ 
and its image by $\Phi$ is trivial in $\Acy_n$. 
Then we can apply the same argument as Case 1 to $M_{0.5}$. 

Such a manifold $M_{0.5}$ is given as follows. 
The homomorphism $i_+:F_n \to \pi_1 (M_0)$ induces 
a homomorphism $i_+^{\mathrm{acy}}:\Acy_n \to 
\pi_1 (M_0)^{\mathrm{acy}}$ satisfying $i_+^{\mathrm{acy}} \circ 
\iota_{F_n}=\iota_{\pi_1 (M_0)} \circ i_+$. Similarly we 
have $\Phi^{\mathrm{acy}}:\pi_1 (M_0)^{\mathrm{acy}} \to \Acy_n$ 
satisfying $\Phi=\Phi^{\mathrm{acy}} \circ \iota_{\pi_1 (M_0)}$. Then 
\[\Phi^{\mathrm{acy}} \circ i_+^{\mathrm{acy}} \circ \iota_{F_n} 
= \Phi^{\mathrm{acy}} \circ \iota_{\pi_1 (M_0)} \circ i_+ = 
\iota_{F_n}.\] 
By the universality of the acyclic closure, we have 
$\Phi^{\mathrm{acy}} \circ i_+^{\mathrm{acy}} 
= \Id_{\Acy_n}$. In particular, $\Phi^{\mathrm{acy}}$ is onto.

Take a simple closed curve $C$ representing $\alpha \in 
\Ker \Phi$. Since $\Phi (\alpha)=0 \in H_1 (\Acy_n)$, 
we can write $\Phi (C)=\prod_{i=1}^{l} [h_{i1},h_{i2}]$ 
with $h_{ij} \in \Acy_n$. 
We take an acyclic system 
\[S \, : \, x_i=w_i (x_1,x_2,\ldots,x_m) \quad (i=1,2,\ldots,m)\]
over $\pi_1 (M_0)$ whose solution in $\pi_1 (M_0)^{\mathrm{acy}}$ 
includes 
\[\{ i_+^{\mathrm{acy}}(h_{11}), i_+^{\mathrm{acy}}(h_{12}),\ldots, 
i_+^{\mathrm{acy}}(h_{l1}), i_+^{\mathrm{acy}}(h_{l2}) \}.\]

We attach a $1$-handle to $M_0 \times \{ 1 \} \subset 
\partial (M_0 \times [0,1])$ 
for each variable $x_i$ and write $x_i$ again for the added 
generator on the fundamental group of the resulting cobordism. 
We also attach a $2$-handle along 
the loop $x_i w_i^{-1}$ for each $i=1,2,\ldots,m$ with any framing. 
We denote the resulting cobordism by $W_{0.5}$. Then 
\[\pi_1 (W_{0.5})=(\pi_1 (M) \ast \langle x_1,x_2,\ldots,x_m \rangle) 
\big/ \langle\!\langle x_1 w_1^{-1}, x_2 w_2^{-1}, \ldots,x_m w_m^{-1} \rangle\!\rangle.\]
We define a homomorphism 
$\Phi_S: \pi_1 (W_{0.5}) \to \pi_1 (M_0)^{\mathrm{acy}}$ which lifts 
$\iota_{\pi_1 (M_0)}$ by sending $x_i$ to 
the corresponding solution of $S$. 
The composite $\Phi^{\mathrm{acy}} \circ \Phi_S:
\pi_1 (W_{0.5}) \to \Acy_n$ 
induces a continuous map $\Phi :W_{0.5} \to K(\Acy_n,1)$ which 
extends $\Phi:M_0 \to K(\Acy_n,1)$. 

The relative chain complex $C_\ast (W_{0.5}, M_0 \times [0,1])$ given by 
the handle decomposition has 
its non-trivial part in degree $1$ and $2$ generated by the above 
newly added handles. The acyclicity of the system $S$ says that 
the boundary of the $2$-handle associated with 
the relation $x_i w_i^{-1}$ is of the form 
\[[x_i]+(\mbox{1-handles in $M_0 \times [0,1]$}) \equiv [x_i] \in 
C_1 (W_{0.5}, M_0\times [0,1]).\]
Therefore $H_\ast (W_{0.5}, M_0 \times [0,1])\cong 
H_\ast (W_{0.5},M_0) =0$ holds. 

Consider $W_{0.5}$ to be a cobordism between 
$M_0=M_0 \times \{0\}$ and a new manifold $M_{0.5}$. 
The dual handle decomposition of $W_{0.5}$ is obtained 
from $M_{0.5} \times [0,1]$ by attaching 
$k$- and $(k+1)$-handles. This shows that the inclusion 
$M_{0.5} \hookrightarrow W_{0.5}$ induces an isomorphism 
$\pi_1 (M_{0.5}) \cong \pi_1 (W_{0.5})$. 
By the Poincar\'e-Lefschetz duality, 
$H_\ast (W_{0.5},M_{0.5}) \cong H^{(k+2)-\ast} (W_{0.5},M_0) =0$. 
Therefore we see that $M_0$ and $M_{0.5}$ are homology 
bordant over $K(\Acy_n,1)$ by 
the bordism $W_{0.5}$ and $\Phi$. Note that this bordism preserves 
the direct sum decomposition 
$H_1 (M_0) \cong i_+(H_1 (X_n^k-\mathrm{Int}\, D^k)) 
\oplus \Ker \Phi$, namely 
we also have $H_1 (M_{0.5}) \cong i_+(H_1 (X_n^k-\mathrm{Int}\, D^k)) 
\oplus \Ker \Phi$ 
and we can take $\overline{\alpha} \in \Ker \Phi \subset H_1 (M_{0.5})$ 
which corresponds to $\alpha$. 

Recall the simple closed curve $C$ taken at the beginning of 
this argument. Since $\pi_1 (M_{0.5}) \to \pi_1 (W_{0.5})$ is an isomorphism, 
there exists a simple closed curve $\overline{C} \subset M_{0.5}$ 
which attains $C$ in $\pi_1 (W_{0.5})$. Now $h_{ij} \in \Acy_n$ are 
in the image of $\Phi_S:\pi_1 (W_{0.5}) \to \Acy_n$, so that we 
can take $\overline{h_{ij}} \in \pi_1 (M_{0.5})$ attaining $h_{ij}$. 
Then the simple closed curve $\overline{C} \left( 
\prod_{i=1}^l [\overline{h_{i1}},\overline{h_{i2}}] \right)^{-1}$ 
represents $\overline{\alpha}$ and is mapped by $\Phi$ 
to the trivial element of $\Acy_n$. The manifold $M_{0.5}$ and 
the map $\Phi: M_{0.5} \to K(\Acy_n,1)$ are what we are looking for 
in this case.

\bigskip
By iterating the above procedure, we succeed in killing 
$\Ker \Phi \subset H_1 (M_0)$ with 
keeping $f=\Phi|_{X_{2n}^k}: X_{2n}^k=\partial M_0 \to K(\Acy_n,1)$ unchanged. 
That is,  we get a manifold $M_n$ 
which is bordant to $M_0$ over $K(\Acy_n,1)$ by a bordism $W_n$ 
and a map $\Phi:W_n \to K(\Acy_n,1)$ such that 
the kernel of $\Phi|_{M_n}: H_1 (M_n) \to H_1(K(\Acy_n,1))$ is trivial. 
Since $\Phi \circ i_+=\iota_{F_n}$, 
the maps 
$i_+:H_1 (X_n^k-\mathrm{Int}\, D^k) \to H_1 (M_n)$ and 
$\Phi:H_1 (M_n) \to H_1 (K(\Acy_n,1))$ are isomorphisms. 

Let us show that $(M_n,i_+,i_-)$ is a homology cylinder over 
$X_n^k-\mathrm{Int}\, D^k$. 
The bordism $W_n$ is obtained from $M_0 \times [0,1]$ by attaching 
$1$-and $2$-handles with the number of $2$-handles greater than that of 
$1$-handles by $n$. 
The dual handle decomposition is obtained from $M_n \times [0,1]$ by attaching 
their dual $k$-and $(k+1)$-handles. Therefore we have 
\[\chi (M_n)+(-1)^k n = \chi (W_n) = \chi (M_0)+n =1-n,\]
where $\chi (\cdot)$ denotes the Euler characteristic, and 
\[H_i (M_n) \cong H_i (W_n) \cong H_i(M_0) =0\]
if $2 \le i \le k-2$. 
Since $M_n$ is a compact $(k+1)$-dimensional manifold with non-empty boundary, 
it is homotopy equivalent to a $k$-dimensional CW-complex. Hence 
$H_i (M_n) =0$ for $i \ge k+1$ and $H_k (M_n)$ is free. 
The inclusion $i_+:X_n^k-\mathrm{Int}\, D^k \hookrightarrow M_n$ 
is decomposed to 
$X_n^k-\mathrm{Int}\, D^k \hookrightarrow \partial M_n 
\hookrightarrow M_n$, which shows 
that $H_1 (M_n,\partial M_n)=0$. By the Poincar\'e-Lefschetz duality, 
we have $H^k (M_n) \cong H_1 (M_n,\partial M_n)=0$. It follows that 
$H_k (M_n)=0$ and $H_{k-1} (M_n)$ is free. Comparing with $\chi (M_n)=
1-n+(-1)^{k-1} n$, we have $H_{k-1} (M_n) \cong \mathbb{Z}^n$. 
Finally we check that $i_+:H_{k-1} (X_n^k-\mathrm{Int}\, D^k) 
\to H_{k-1} (M_n)$ is an 
isomorphism. The source and target are both 
isomorphic to $\mathbb{Z}^n$. 
The homomorphism $\varphi \circ \iota_{F_n}:\pi_1 (X_n^k-\mathrm{Int}\, D^k) 
\to \Acy_n$ is 
$2$-connected and 
factors through $\pi_1 (X_n^k-\mathrm{Int}\, D^k) 
\xrightarrow{i_-} \pi_1 (M_n) 
\xrightarrow{\Phi} \Acy_n$. From this, we see that 
$H_1 (M_n, i_-(X_n^k-\mathrm{Int}\, D^k))=0$. 
The Poincar\'e-Lefschetz duality shows that 
\begin{align*}
H_k (M_n,i_+(X_n^k-\mathrm{Int}\, D^k)) 
&\cong H^1 (M_n, i_-(X_n^k-\mathrm{Int}\, D^k))\\
&\cong 
\Hom (H_1 (M_n, i_-(X_n^k-\mathrm{Int}\, D^k)), 
\mathbb{Z})=0
\end{align*}
\noindent
and it follows that $H_{k-1} (M_n,i_+(X_n^k-\mathrm{Int}\, D^k))$ 
is a torsion module. 
In fact, it is trivial because the universal coefficient theorem and 
the Poincar\'e-Lefschetz duality say that 
\[H_{k-1} (M_n,i_+(X_n^k-\mathrm{Int}\, D^k)) \cong 
H^k (M_n,i_+(X_n^k-\mathrm{Int}\, D^k)) 
\cong H_1 (M_n, i_-(X_n^k-\mathrm{Int}\, D^k)) =0.\]
Hence $M:=(M_n,i_+,i_-)$ is a homology cylinder over 
$X_n^k-\mathrm{Int}\, D^k$. Now we have a commutative diagram:

\medskip
\hspace{90pt}
\SelectTips{cm}{}
\xymatrix{
X_n^k-\mathrm{Int}\, D^k 
\ar[r]^{f_-} 
\ar[dr]_{i_-} & K(\Acy_n,1) & 
X_n^k-\mathrm{Int}\, D^k 
\ar[l]_{f_+} 
\ar[dl]^{i_+} \\
& M_n \ar[u]^\Phi & 
}

\medskip
\noindent
From this diagram, we see that 
\[\mathrm{Acy}(M)=(i_+^{\mathrm{acy}})^{-1} \circ 
i_-^{\mathrm{acy}}=(f_+^{\mathrm{acy}})^{-1} \circ 
f_-^{\mathrm{acy}}=\varphi.\]
This completes the proof. 
\end{proof}

\begin{remark}
Theorem \ref{thm:acysurj} is considered to be an analogue of a part of 
Laudenbach's theorem \cite[Th\'eor\`eme 4.3]{laudenbach} that the natural action of 
the diffeotopy group $\mathcal{M} (X_n^3-\mathrm{Int}\, D^3)$ on 
$\pi_1 (X_n^3-\mathrm{Int}\, D^3)=F_n$ gives 
an epimorphism $\mathcal{M} (X_n^3-\mathrm{Int}\, D^3) 
\twoheadrightarrow \Aut \, (F_n)$, where 
the same statement holds for $X_n^k-\mathrm{Int}\, D^k$ with $k \ge 4$. 
\end{remark}

\begin{remark}
For homology cylinders over a surface $\Sigma_{g,1}$, 
we can also define a homomorphism 
$\mathrm{Acy}: \mathcal{H}_{g,1} \to \Aut \,(\AC{F_{2g}})$, 
where $F_{2g}= \pi_1 (\Sigma_{g,1})$. 
In this case, however, it was shown in 
\cite[Theorem 6.1]{sa2} that $\mathrm{Acy}$ is {\it not\/} 
surjective. In fact, the image is given by
\[\{ \varphi \in \Aut (\AC{F_{2g}}) \mid
\varphi(\zeta) = \zeta \in \AC{F_{2g}} \},\]
where $\zeta \in F_{2g} \subset \AC{F_{2g}}$ 
is the word corresponding to the boundary loop of $\Sigma_{g,1}$. 
This may be regarded as an analogue of the Dehn-Nielsen theorem 
for the action of the mapping class group $\mathcal{M}_{g,1}$ on 
$F_{2g}$. 
\end{remark}

%----------The Magnus representation revisited--------------

\section{The Magnus representation revisited}\label{sec:magnus2}

Now we give an alternative description of the Magnus representation 
for homology cylinders and use it to finish the proof of 
Theorem \ref{thm:acyMag}. 
For that, we recall the extended free derivatives 
originally defined by Le Dimet \cite{ld}. 
Precisely speaking, the derivatives given below are 
a reduced version to commutative rings. 

Let $\{\gamma_1, \gamma_2, \ldots, \gamma_n\}$ be a basis of a 
free group $F_n$. 
The definition of the extended free derivatives is derived 
from the following lemma. The proof is almost the same as that 
of \cite[Proposition 1.1]{ld}. 

\begin{proposition}\label{extfree} 
The homomorphism
\[\chi: \mathcal{K}_{H_1}^n \longrightarrow 
I(\Acy_n) \otimes_{\Z [\Acy_n]} \mathcal{K}_{H_1}\]
sending $(a_1,\ldots,a_n)^T \in \mathcal{K}_{H_1}^n$ to \ 
$\displaystyle\sum_{i=1}^n (\gamma_i^{-1}-1) \otimes a_i$ 
is a right $\mathcal{K}_{H_1}$-isomorphism, where 
$I(\Acy_n)$ is the kernel of the trivializer 
$\mathbb{Z} [\Acy_n] \to \mathbb{Z}$.
\end{proposition}

\begin{definition}
For $1 \le i \le n$, the {\it extended free derivative} 
\[\frac{\partial}{\partial \gamma_i}:\Acy_n 
\longrightarrow \mathcal{K}_{H_1}\]
with respect to $\gamma_i$ is the map assigning to 
$v \in \Acy_n$ the $i$-th component of \,
$\overline{\chi^{-1} ((v^{-1}-1) \otimes 1)} \in \mathcal{K}_{H_1}$.
\end{definition}

In Le Dimet \cite[Proposition 1.3]{ld}, the formulas 
\[\frac{\partial \gamma_j}{\partial \gamma_i}=\delta_{i,j}, \qquad 
\frac{\partial (gh)}{\partial \gamma_i}=
\frac{\partial g}{\partial \gamma_i}+g\frac{\partial h}{\partial \gamma_i}, 
\qquad 
\frac{\partial g^{-1}}{\partial \gamma_i}=-g^{-1} 
\frac{\partial g}{\partial \gamma_i}\]
for $g,h \in \Acy_n$ are given. 
By them, we see that 
the extended free derivatives 
coincide with the original ones 
if we restrict them to $F_n$. 
\begin{definition}\label{def:alternative}
The {\it Magnus representation} for $\Aut (\Acy_n)$ is the map 
\[r:\Aut (\Acy_n) \to M(n,\mathcal{K}_{H_1})\]
assigning to $\varphi \in \Aut (\Acy_n)$ the matrix 
\[r(\varphi):=\left(\overline{\left(\frac{\partial \varphi(\gamma_j)}
{\partial \gamma_i} \right)}\right)_{i,j}.\]
\end{definition}
It is not difficult to check that 
the Magnus representation $r$ is a crossed homomorphism and 
hence the image of $r$ is included in the set 
$\mathrm{GL}(n,\mathcal{K}_{H_1})$. 
When we compute the Magnus matrix $r(\mathrm{Acy} (M))$ 
for a homology cylinder $M \in \mathcal{H} (X_n^k -\mathrm{Int}\, D^k)$, 
we shall meet the same formula 
as Proposition \ref{prop:MagnusFormula}. This shows that 
Definition \ref{def:alternative} gives an alternative 
description of the Magnus representation. 
\begin{example}\label{ex:twoconn}
Let $f$ be a $2$-connected endomorphism of $F_n$. 
By Proposition \ref{prop:isom}, the endomorphism 
$f$ is uniquely extended to an automorphism $f^{\mathrm{acy}}$ of 
$\Acy_n$. In this case, the Magnus matrix $r(f^{\mathrm{acy}})$ 
is obtained by applying the original free derivatives to $f$. 
\end{example}

Consider the composition
\[\widetilde{r}:
\Aut (\Acy_n) 
\stackrel{r}{\longrightarrow} 
\mathrm{GL}(n,\mathcal{K}_{H_1}) 
\xrightarrow{\det} 
\mathcal{K}_{H_1}^\times 
\longrightarrow 
\mathcal{K}_{H_1}^\times/(\pm H_1 \cdot A') \cong \Z^\infty,\]
where $A':=\{f^{-1} \cdot \varphi (f) \mid 
f \in \mathcal{K}_{H_1}^\times, \ 
\varphi \in \Aut (H_1)\}$ as before. 
The map $\widetilde{r}$ 
is a homomorphism. 

\begin{theorem}\label{thm:H1Acy}
For any $n \ge 2$, we have:
\begin{itemize}
\item[$(1)$] The image of the homomorphism $\widetilde{r}$ 
is an infinitely generated subgroup of $\mathbb{Z}^\infty$. In particular, 
$H_1 (\mathrm{Aut}\, (F^{\mathrm{acy}}_n))$ has 
infinite rank. 

\item[$(2)$] The homomorphism $\widetilde{r}$ factors through 
$\mathrm{Out}\, (F^{\mathrm{acy}}_n)$. Therefore 
$H_1 (\mathrm{Out}\, (F^{\mathrm{acy}}_n))$ has infinite rank. 
\end{itemize}
\end{theorem}
\begin{proof}
$(1)$ Consider a homomorphism $f_m: F_n \to F_n$ 
defined by 
\[f_m(\gamma_1) = 
(\gamma_1 \gamma_2^{-1} \gamma_1^{-1} \gamma_2^{-1})^m 
\gamma_1 \gamma_2^{2m}, \qquad
f_m(\gamma_i) = \gamma_i \ (2 \le i \le n) \]
for each $m \ge 1$. The homomorphism $f_m$ is $2$-connected 
and therefore it induces an automorphism $f_m^{\mathrm{acy}}$ 
of $F^{\mathrm{acy}}_n$. 
By using the original free derivatives (see Example \ref{ex:twoconn}), 
we see that the Magnus matrix $r(f_m^{\mathrm{acy}})$ is given by 
the lower triangular matrix 
\[\begin{pmatrix}
\begin{array}{cc}
1 -\gamma_2 +\gamma_2^2-\gamma_2^3+ \cdots 
+ \gamma_2^{2m} & 0 \\
\ast & 1 
\end{array} & 0_{(2,n-2)}\\
\qquad \quad 0_{(n-2,2)} & I_{n-2}
\end{pmatrix}.\] 
Therefore we have 
\begin{align*}
\widetilde{r} (f_m^{\mathrm{acy}}) &= 
1 -\gamma_2 +\gamma_2^2-\gamma_2^3+ \cdots 
+ \gamma_2^{2m}\\
&=1 +(-\gamma_2) +(-\gamma_2)^2+(-\gamma_2)^3+ \cdots 
+ (-\gamma_2)^{2m}.
\end{align*}
\noindent
By a well known fact on the cyclotomic 
polynomials, we see that 
the polynomial $\widetilde{r} (f_m^{\mathrm{acy}})$ is 
irreducible if $2m+1$ is prime. 
Moreover, the polynomials 
$\{\widetilde{r} (f_m^{\mathrm{acy}}) \mid 
\mbox{$2m+1$ is prime}\}$ 
are independent in the module 
$\mathcal{K}_{H_1}^\times/(\pm H_1\cdot A')$ because 
their degrees are distinct. 
Therefore the claim for $H_1 (\mathrm{Aut}\, (F^{\mathrm{acy}}_n))$ 
follows. 

$(2)$ It suffices to show that the composition
\[\Psi: F_n^{\mathrm{acy}} \twoheadrightarrow 
\mathrm{Inn}\, (F^{\mathrm{acy}}_n) 
\hookrightarrow 
\mathrm{Aut}\, (F^{\mathrm{acy}}_n) 
\stackrel{\widetilde{r}}{\longrightarrow} 
\mathcal{K}_{H_1}^\times/(\pm H_1\cdot A')\]
is trivial. The restriction of the homomorphism $\Psi$ to $F_n$ 
is trivial because 
the determinant of the Magnus matrix of an automorphism of $F_n$ is in $H_1$. 
Therefore $\Psi$ is an extension of the trivial map from $F_n$ to 
the abelian group $\mathcal{K}_{H_1}^\times/(\pm H_1\cdot A')$ which is AC. 
Then by Proposition \ref{prop:universality}, 
we see that the map $\Psi$ is also trivial. 
\end{proof}
\begin{proof}[Proof of Theorem $\ref{thm:acyMag}$] 
The Magnus representation mentioned in Section \ref{sec:higher} 
is the composition of the homomorphism $\mathrm{Acy}$ and the above $\widetilde{r}$. 
Therefore our claims immediately follow from Theorems \ref{thm:acysurj} 
and \ref{thm:H1Acy}. 
\end{proof}

The isomorphisms $f_m$ in the proof of Theorem \ref{thm:H1Acy} can also 
be used to show the following.  
\begin{theorem}\label{thm:infinite}
The acyclic closure $F_n^{\mathrm{acy}}$ is not finitely 
generated for any $n \ge 2$. 
\end{theorem}
\begin{proof}
Suppose $F_n^{\mathrm{acy}}$ had a finite generating set 
$\{g_1, g_2,\ldots,g_l\}$. 
Then the above formulas for the extended free derivatives 
imply that the image of the derivative 
$\frac{\partial }{\partial \gamma_i}$ 
for each $i$ is in the subring $R$ 
of $\mathcal{K}_{H_1}$ obtained from $\mathbb{Z}[H_1]$ by 
adding $\left\{ \frac{\partial g_1}{\partial \gamma_i}, 
\frac{\partial g_2}{\partial \gamma_i}, 
\ldots, \frac{\partial g_l}{\partial \gamma_i}\right\}$. 
In particular, there are only finitely many irreducible polynomials 
which appear as factors of the denominators of 
reduced expressions for elements in $R$. However, for 
the automorphism $f_m^{\mathrm{acy}}$ with $2m+1$ prime 
constructed in the proof of Theorem \ref{thm:H1Acy}, 
the $(1,1)$-entry of $\widetilde{r}((f_m^{\mathrm{acy}})^{-1})=
(\widetilde{r}(f_m^{\mathrm{acy}}))^{-1}$ is 
\[\overline{\left(\frac{\partial (f_m^{\mathrm{acy}})^{-1}(\gamma_1)}
{\partial \gamma_1} \right)}=
\frac{1}{1 -\gamma_2 +\gamma_2^2-\gamma_2^3+ \cdots 
+ \gamma_2^{2m}}.\]
This contradicts to the property of $R$ just mentioned. 
\end{proof}
\begin{remark}
It is easy to see that the derivative 
$\displaystyle\frac{\partial}{\partial \gamma_i}:\Acy_n 
\to \mathcal{K}_{H_1}$ 
factors through the metabelian quotient 
\[\Acy_n/[[\Acy_n,\Acy_n],[\Acy_n,\Acy_n]]\]
of $\Acy_n$. Our proof of Theorem \ref{thm:infinite} shows 
that this metabelian quotient is also infinitely generated 
for any $n \ge 2$, and therefore 
it is not isomorphic to that of $F_n$. This fact contrasts with 
the nilpotent quotients of $F_n$ and $\Acy_n$ 
which are isomorphic by Stallings' theorem. 
\end{remark}

As mentioned in Section \ref{sec:acyclic}, 
the acyclic closure of a group is a variation of the algebraic closure of a group.  
The argument in this section can be applied to 
the algebraic closure $\Alg_n$ of a free group $F_n$. 
In fact, as shown in \cite[Proposition 5]{le1},  
an automorphism of $\Alg_n$ is induced from a {\it normally surjective} $2$-connected 
endomorphism of $F_n$. The $2$-connected endomorphisms $f_m$ of $F_n$ 
constructed in the proof of Theorem \ref{thm:H1Acy} are normally surjective. Indeed,  
$\gamma_1=(\gamma_1 \gamma_2^{-1} \gamma_1^{-1} \gamma_2^{-1})^{-m}  
f_m(\gamma_1) \gamma_2^{-2m}$ is in the normal closure of the image of $f_m$ 
because $\gamma_1 \gamma_2^{-1} \gamma_1^{-1} \gamma_2^{-1}=
(\gamma_1 \gamma_2^{-1} \gamma_1^{-1}) \gamma_2^{-1}$ is in it. Thus $f_m$ induces 
an automorphism $f_m^{\mathrm{alg}}$ of $\Alg_n$. 
By Le Dimet's original construction, we can define the Magnus representation for 
$\mathrm{Aut}\, (F_n^{\mathrm{alg}})$. 
The remaining argument goes parallel to the case of 
$\Acy_n$. Consequently, we have the following.

\begin{theorem}\label{thm:H1Alg}
For any $n \ge 2$, we have:

$(1)$ $H_1 (\mathrm{Aut}\, (F_n^{\mathrm{alg}}))$ and 
$H_1 (\mathrm{Out}\, (F_n^{\mathrm{alg}}))$ have infinite rank. 

$(2)$ The algebraic closure $F_n^{\mathrm{alg}}$ and its metabelian quotient 
$\Alg_n/[[\Alg_n,\Alg_n],[\Alg_n,\Alg_n]]$ 
are not finitely generated. 
\end{theorem}

%%%%%%%%%%%%%%%%%%%%%%%%%%%%%%%%%%%%%%%%%%%%%%%%
% References 
%%%%%%%%%%%%%%%%%%%%%%%%%%%%%%%%%%%%%%%%%%%%%%%%

\end{document}